\documentclass[12pt]{amsart}
\usepackage{amscd, amsfonts, amssymb, mathrsfs}
\usepackage{fullpage}
\usepackage{latexsym}
\usepackage{verbatim}
\usepackage{enumerate}
\usepackage[all,cmtip]{xy}
\usepackage[colorlinks=true]{hyperref}

\newcommand{\iso}{\cong}
\numberwithin{equation}{section}

\newtheorem{thm}[equation]{Theorem}

\newtheorem{lem}[equation]{Lemma}
\newtheorem{cor}[equation]{Corollary}
\newtheorem{prop}[equation]{Proposition}

\theoremstyle{definition}
\newtheorem{defn}[equation]{Definition}
\newtheorem{ex}[equation]{Example}

\theoremstyle{remark}
\newtheorem{rem}[equation]{Remark}
\theoremstyle{remark}
\newtheorem{rems}[equation]{Remarks}

\newcommand{\GL}{{\rm GL}}

\renewcommand{\Im}{{\rm Im}}

\newcommand{\ovl}{\overline}


\title[Semisimplification for subgroups of reductive algebraic groups]
{Semisimplification for subgroups of reductive algebraic groups}

\author[M.\  Bate]{Michael Bate}
\address%[M.\  Bate]
{Department of Mathematics,
University of York,
York YO10 5DD,
United Kingdom}
\email{michael.bate@york.ac.uk}

\author[B.\ Martin]{Benjamin Martin}
\address%[B.\ Martin]
{Department of Mathematics,
University of Aberdeen,
King's College,
Fraser Noble Building,
Aberdeen AB24 3UE,
United Kingdom}
\email{b.martin@abdn.ac.uk}
 
\author[G. R\"ohrle]{Gerhard R\"ohrle}
\address%[G.~R\"{o}hrle]
{Fakult\"at f\"ur Mathematik,
Ruhr-Universit\"at Bochum,
D-44780 Bochum, Germany}
\email{gerhard.roehrle@rub.de}

\begin{document}

\begin{abstract}
 Let $G$ be a reductive algebraic group---possibly non-connected---over a field $k$ and let $H$ be a subgroup of $G$.  If $G= \GL_n$ then there is a degeneration process for obtaining from $H$ a completely reducible subgroup $H'$ of $G$; one takes a limit of $H$ along a cocharacter of $G$ in an appropriate sense.  We generalise this idea to arbitrary reductive $G$ using the notion of $G$-complete reducibility and results from geometric invariant theory over non-algebraically closed fields due to the authors and Herpel.  Our construction produces a $G$-completely reducible subgroup $H'$ of $G$, unique up to $G(k)$-conjugacy, which we call a {\em $k$-semisimplification} of $H$.  This gives a single unifying construction which extends various special cases in the literature (in particular, it agrees with the usual notion for $G= \GL_n$ and with Serre's ``$G$-analogue'' of semisimplification for subgroups of $G(k)$ from \cite{serre2}).  We also show that under some extra hypotheses, one can pick $H'$ in a more canonical way using the Tits Centre Conjecture for spherical buildings and/or the theory of optimal destabilising cocharacters introduced by Hesselink, Kempf and Rousseau.
\end{abstract}

\maketitle

\section{Introduction}
\label{sec:intro}

The aim of this paper is to present a construction of the \emph{semisimplification} of a subgroup $H$ of a (possibly non-connected) reductive linear algebraic group $G$ over an arbitrary field $k$.
This construction unifies and generalizes many concepts already in the literature within a single framework. 
For example, the semisimplification of a module for a group is a well-known construction in representation theory, 
corresponding in our case to the situation where $H\subseteq \GL_n(k)$. 
Building on this idea, for $G$ a connected reductive linear algebraic group over a field $k$ and $H$ a subgroup of $G(k)$, Serre introduced the concept of a ``$G$-analogue" of semisimplification from representation theory in \cite[\S 3.2.4]{serre2}.
This notion is also used for representations of various kinds of algebras, e.g., see \cite{kraft}, \cite{BHRZ}, \cite{riedtmann}, \cite{zwara1}, and \cite{zwara2}.
It is also an ingredient in work of Lawrence-Sawin on the Shafarevich Conjecture for abelian varieties
\cite{LawrenceSawin} and work of Lawrence-Venkatesh on Mordell's Conjecture \cite{LawrenceVenkatesh},
which involve Galois representations taking values in possibly non-connected reductive $p$-adic groups.

We begin by recalling how the most basic case works.
Let $n\in {\mathbb N}$ and let $H$ be a subgroup of $\GL_n(k)$.
There is an $H$-module filtration of $k^n$ such that the successive quotients are irreducible, by the Jordan-H\"older Theorem.
In terms of matrices, this implies that, by changing basis if necessary, we may assume that $H$ is in upper block-triangular form, with the action of $H$ on each quotient being represented by the corresponding block on the diagonal. Letting $H'$ be the subgroup of $\GL_n(k)$ consisting of the block diagonal matrices obtained by taking each element of $H$ and replacing the entries above the block diagonal with 0s, we obtain a subgroup which acts semisimply on $k^n$---that is, $H'$ is completely reducible.  
Since this construction is independent of the choice of basis up to $\GL_n(k)$-conjugacy, again by the Jordan-H\"older Theorem, it is therefore reasonable to call $H'$ the \emph{semisimplification} of $H$. 

We now explain some of the ingredients of our construction in the case that $k$ is algebraically closed, which removes some technicalities.  
Recall \cite{BMR}, \cite{serre2} that if $G$ is connected and $H$ is a subgroup of $G$ then $H$ is \emph{$G$-completely reducible} ($G$-cr for short) if for any parabolic subgroup $P$ of $G$ such that $P$ contains $H$, there is a Levi subgroup $L$ of $P$ such that $L$ contains $H$.  If $G= \GL_n$ then $H$ is $G$-cr if and only if $k^n$ is completely reducible as an $H$-module; this follows from the usual characterisation of parabolic subgroups of $\GL_n$ as stabilizers of flags of subspaces.  We make the same definition for arbitrary reductive $G$, replacing parabolic subgroups and Levi subgroups with R-parabolic subgroups and R-Levi subgroups instead (see Section~\ref{sec:cochar} for details).

To perform our construction, we apply a characterisation of $G$-complete reducibility in terms of geometric invariant theory (GIT).  We see this idea already in our original example: we can view $H'$ as a degeneration of $H$ in the following sense.  
Let the sizes of the blocks down the diagonal be $n_1,\ldots,n_r$, and 
define a cocharacter $\lambda\colon {\mathbb G}_m\to \GL_n$ by
%$\lambda(a)= {\rm diag}(a^{-1},\ldots, a^{-1}, \ldots, a^{-r},\ldots, a^{-r})$ ($n_i$ occurrences of $a^{-i}$)
$$
\lambda(a)= {\rm diag}(a^r,\ldots, a^r, \ldots, a^1,\ldots, a^1), \text{ with }  n_i \text{ occurrences of } a^{r-i+1}, 1\leq i \leq r.
$$
For each $a\in k^*$, define $H_a= \lambda(a)H\lambda(a)^{-1}$ for $a\in k^*$.  Then $H'=
\lim_{a\to 0} H_a$ in an appropriate sense.

Our definition of $k$-semisimplification (Definition~\ref{defn:ss}) for arbitrary $k$ is new, 
generalizes the one given by Serre in \cite[\S 3.2.4]{serre2}, and is closely related to the definition given in \cite{GIT} using optimal destabilising cocharacters; the two notions agree whenever the latter makes sense (cf.\ also \cite[Sec.\ 4]{genred} for the algebraically closed case).  We prove that the $k$-semisimplification of a subgroup $H$ of $G$ is unique up to conjugacy (Theorem~\ref{thm:main}), generalizing \cite[Prop.~3.3(b)]{serre2}.  
In Theorem~\ref{thm:normal} we show that a normal subgroup of a $G$-completely reducible subgroup $H$ is $G$-completely reducible and that the process of $k$-semisimplification behaves well under passing to normal subgroups of $H$, if $k$ is perfect or $G$ is connected.  The proof rests on deep results from the theory of spherical buildings and the Hesselink-Kempf-Rousseau theory of optimal destabilising cocharacters.  We give a short and self-contained exposition, bringing together some results (such as Corollary~\ref{cor:char0}) that follow from previous work but are not easily extracted from earlier papers.

\section{Cocharacter-closed orbits}
\label{sec:cochar}

Following \cite{borel} and our earlier work \cite{GIT}, \cite{cochar}, we regard an affine variety over a field $k$ as a variety $X$ over the algebraic closure $\ovl{k}$ together with a choice of $k$-structure.  We denote the separable closure of $k$ by $k_s$.  We write $X(k)$ for the set of $k$-points of $X$ and $X(\ovl{k})$ (or just $X$) for the set of $\ovl{k}$-points of $X$.  By a subvariety of $X$ we mean a closed $\ovl{k}$-subvariety of $X$; a $k$-subvariety is a subvariety that is defined over $k$.  We denote by $M_n$ the associative algebra of $n\times n$ matrices over $k$.  Below $G$ denotes a possibly non-connected reductive linear algebraic group over $k$.  
By a subgroup of $G$ we mean a closed $\ovl{k}$-subgroup and by a $k$-subgroup we mean a subgroup that is defined over $k$. 
(We note here that much of what follows works for non-closed subgroups---most of the important conditions hold for $H$ if and only if they hold for the Zariski closure $\overline{H}$; the details are left to the reader.) By $G^0$ we denote the identity component of $G$, and likewise for subgroups of $G$.
	
	We define $Y_k(G)$ to be the set of $k$-defined cocharacters of $G$
and $Y(G) := Y_{\ovl{k}}(G)$ to be the set of all  cocharacters of $G$.

 Let $H$ be a subgroup of $G$.  Even if $H$ is $k$-defined, the (set-theoretic) centralizer $C_G(H)$ need not be $k$-defined in general.  It is useful to have criteria to ensure that $C_G(H)$ is $k$-defined (see Proposition~\ref{prop:ratgeom} and Section~\ref{sec:opt}).  For instance, if $k$ is perfect and $H$ is $k$-defined then $C_G(H)$ is $k$-defined.  We say that $H$ is \emph{separable} if the scheme-theoretic centralizer ${\mathscr C}_G(H)$ is smooth \cite[Def.~3.27]{BMR}; for instance, any  subgroup of $\GL_n$ is separable \cite[Ex.~3.28]{BMR} (see \cite{BMRT} for more examples of separable subgroups).  If $H$ is $k$-defined and separable then $C_G(H)$ is $k$-defined (see \cite[Prop.\ 7.4]{cochar}).
 
Next we recall some basic notation and facts 
concerning 
parabolic subgroups in (non-connected) reductive groups $G$ from \cite[\S 6]{BMR} and \cite{GIT}.
Given $\lambda\in Y(G)$, we define 
$$
P_\lambda= \{g\in G\,|\,\lim_{a\to 0} \lambda(a)g\lambda(a)^{-1}\ \mbox{exists}\}
$$
and $L_\lambda= C_G({\rm Im}(\lambda))$ (for the definition of a limit, see \cite[Sec.~3.2.13]{springer}).  We call $P_\lambda$ an \emph{R-parabolic subgroup} of $G$ and $L_\lambda$ an \emph{R-Levi subgroup} of $P_\lambda$; they are subgroups of $G$.  We have $P_\lambda= L_\lambda= G$ if $\Im(\lambda)$ is contained in the centre of $G$. 
For ease of reference, we record without proof some basic facts about these subgroups.

\begin{lem} 
\begin{itemize}
%\item[(i)] If $P$ is an R-parabolic subgroup of $G$ then $P$ is a parabolic subgroup in the sense that $G/P$ is a complete variety. 
\item[(i)] If $P$ is a $k$-defined R-parabolic subgroup then $R_u(P)$ is $k$-defined.
\item[(ii)] If $P$ is a parabolic subgroup of $G^0$ then the normalizer $N_G(P)$ is an R-parabolic subgroup of $G$, and $N_G(P)$ is $k$-defined if $P$ is.  
\end{itemize}
\end{lem}

If $G$ is connected then every pair $(P,L)$ consisting of a parabolic $k$-subgroup $P$ of $G$ and a Levi $k$-subgroup $L$ of $P$ is of the form $(P,L)= (P_\lambda, L_\lambda)$ for some $\lambda\in Y_k(G)$, and vice versa \cite[Lem.~15.1.2(ii)]{springer}.  
In general, if $\lambda\in Y_k(G)$ then $P_\lambda$ and $L_\lambda$ are $k$-defined \cite[Lem.~2.5]{GIT}, but the converse is not so straightforward. 
If $P$ is an R-parabolic $k$-subgroup and $L$ is an R-Levi $k$-subgroup of $P$ then for any maximal $k$-torus $T$ of $L$, there exists $\lambda\in Y_{k_s}(T)$ such that $P= P_\lambda$ and $L= L_\lambda$.  
However, it is possible that $P$ is a $k$-defined R-parabolic subgroup and yet there does not exist any $\mu\in Y_k(G)$ such that $P= P_\mu$, and similarly for R-Levi subgroups---see \cite[Rem.~2.4]{GIT}.  This complicates some of the arguments below. 

\begin{lem}
\label{lem:parprops}
Let $P$ be an $R$-parabolic subgroup of $G$ and $L$ an R-Levi subgroup of $P$.
\begin{itemize}
\item[(i)] We have $P\iso L\ltimes R_u(P)$, and this is a $k$-isomorphism if $P$ and $L$ are $k$-defined.  
\item[(ii)] Any two R-Levi $k$-subgroups of an R-parabolic $k$-subgroup $P$ are $R_u(P)(k)$-conjugate.  
\end{itemize}
\end{lem}

We denote the canonical projection from $P$ to $L$ by $c_L$; this is $k$-defined if $P$ and $L$ are.  
If we are given $\lambda\in Y(G)$ such that $P= P_\lambda$ and $L= L_\lambda$ then we often write $c_\lambda$ instead of $c_L$.  We have $c_\lambda(g)= \lim_{a\to 0} \lambda(a)g\lambda(a)^{-1}$ for $g\in P_\lambda$; the kernel of $c_\lambda$ is the unipotent radical $R_u(P_\lambda)$ and the set of fixed points of $c_\lambda$ is $L_\lambda$. 

Let $m\in {\mathbb N}$.  Below we consider the action of $G$ on $G^m$ by simultaneous conjugation: $g\cdot (g_1,\ldots, g_m)= (gg_1g^{-1},\ldots, gg_mg^{-1})$.  Given $\lambda\in Y(G)$, we have a map $P_\lambda^m\to L_\lambda^m$ given by ${\mathbf g}\mapsto \lim_{a\to 0} \lambda(a)\cdot {\mathbf g}$; we abuse notation slightly and also call this map $c_\lambda$.  For any ${\mathbf g}\in P_\lambda^m$, there exists an R-Levi $k$-subgroup $L$ of $P_\lambda$ with ${\mathbf g}\in L^n$ if and only if $c_\lambda({\mathbf g})= u\cdot {\mathbf g}$ for some $u\in R_u(P_\lambda)(k)$.

Our main tool from GIT is the notion of cocharacter-closure, introduced in \cite{GIT} and \cite{cochar}.

\begin{defn}
 Let $X$ be an affine $G$-variety and let $x\in X$ (we do not require $x$ to be a $k$-point).  We say that the orbit $G(k)\cdot x$ is \emph{cocharacter-closed over $k$} if for all $\lambda\in Y_k(G)$ such that $x':= \lim_{a\to 0} \lambda(a)\cdot x$ exists, $x'$ belongs to $G(k)\cdot x$.  If $k= \ovl{k}$ then it follows from the Hilbert-Mumford Theorem that $G(k)\cdot x$ is cocharacter-closed over $k$ if and only if $G(k)\cdot x$ is closed \cite[Thm.~1.4]{kempf}.  If ${\mathcal O}$ is a $G(k)$-orbit in $X$ then we say that ${\mathcal O}$ is \emph{accessible from $x$ over $k$} if there exists $\lambda\in Y_k(G)$ such that $x':= \lim_{a\to 0} \lambda(a)\cdot x$ belongs to ${\mathcal O}$.
\end{defn}

\begin{ex}
\label{ex:accessible_tuple}
 If $X= G^m$, $\lambda\in Y_k(G)$ and ${\mathbf g}\in P_\lambda^m$ then $G(k)\cdot c_\lambda({\mathbf g})$ is accessible from ${\mathbf g}$ over $k$.
\end{ex}

The following result is \cite[Thm.\ 1.3]{cochar}.

\begin{thm}[Rational Hilbert-Mumford Theorem]
\label{thm:RHMT}
 Let $G$, $X$, $x$ be as above.  Then there is a unique $G(k)$-orbit ${\mathcal O}$ such that ${\mathcal O}$ is cocharacter-closed over $k$ and accessible from $x$ over~$k$.
\end{thm}

\section{$G$-complete reducibility}
\label{sec:Gcr}

\begin{defn}
 Let $H$ be a subgroup of $G$.  We say that $H$ is \emph{$G$-completely reducible over $k$} ($G$-cr over $k$) if for any R-parabolic $k$-subgroup $P$ of $G$ such that $P$ contains $H$, there is an R-Levi $k$-subgroup $L$ of $P$ such that $L$ contains $H$.  We say that $H$ is \emph{$G$-irreducible over $k$} ($G$-ir over $k$) if $H$ is not contained in any proper R-parabolic $k$-subgroup of $G$ at all.
\end{defn}

\begin{rem}
\label{rem:absolute}
 We say that $H$ is $G$-cr if $H$ is $G$-cr over $\ovl{k}$---cf.\ Section~\ref{sec:intro}.  More generally, if $k'/k$ is an algebraic field extension then we may regard $G$ as a $k'$-group and it makes sense to ask whether $H$ is $G$-cr over $k'$.
\end{rem}

\noindent For more on $G$-complete reducibility, see \cite{serre1.5}, \cite{serre2}, \cite{BMR}.

Note that the definitions make sense even if $H$ is not $k$-defined.  It is immediate that $G$-irreducibility over $k$ implies $G$-complete reducibility over $k$.  We have $P_{g\cdot \lambda}= gP_\lambda g^{-1}$ and $L_{g\cdot \lambda}= gL_\lambda g^{-1}$ for any $\lambda\in Y(G)$ and any $g\in G$ (see, e.g.,  \cite[\S 6]{BMR}).  It follows that if $H$ is $G$-cr over $k$ (resp., $G$-ir over $k$) then so is any $G(k)$-conjugate of $H$.  More generally, one can show that if $H$ is $G$-cr over $k$ (resp., $G$-ir over $k$) then so is $\phi(H)$, for any $k$-defined automorphism $\phi$ of $G$.  If $k= \ovl{k}$ and $H$ is $G$-cr then $H$ is reductive \cite[Prop.\ 4.1]{serre2}, \cite[\S 2.4, \S 6.2]{BMR}.  It follows from Proposition~\ref{prop:ratgeom} below that if $H$ is $k$-defined, $k$ is perfect and $H$ is $G$-cr over $k$ then $H$ is reductive.  We see below (Corollary~\ref{cor:char0}) that the converse holds in characteristic 0.  On the other hand, the converse is false in general, as is shown by the example in  \cite[Proof of Prop.~1.10]{uch1}.

 We now explain the link between $G$-complete reducibility and GIT.  Fix a $k$-embedding $\iota\colon G\to \GL_n$ for some $n\in {\mathbb N}$.  Let $H$ be a subgroup of $G$.  Let $m\in {\mathbb N}$ and let ${\mathbf h}= (h_1,\ldots, h_m)\in H^m$.  We call ${\mathbf h}$ a \emph{generic tuple for $H$ with respect to $\iota$} if $h_1,\ldots, h_m$ generate the subalgebra of $M_n$ generated by $H$ \cite[Def.~5.4]{GIT}.  Note that we don't insist that ${\mathbf h}$ is a $k$-point.  Our constructions below do not depend on the choice of $\iota$, so we suppress the words ``with respect to $\iota$''.  It is immediate that if ${\mathbf h}\in H^m$ is a generic tuple for $H$ and $g\in G$ then $g\cdot {\mathbf h}$ is a generic tuple for $gHg^{-1}$.

\begin{thm}[{\cite[Thm.\ 9.3]{cochar}}]
\label{thm:GIT_crit}
 Let $H$ be a subgroup of $G$ and let ${\mathbf h}\in H^m$ be a generic tuple for $H$.  Then $H$ is $G$-completely reducible over $k$ if and only if $G(k)\cdot {\mathbf h}$ is cocharacter-closed over $k$.
\end{thm}

Using this result one can derive many results on $G$-complete reducibility: for instance, see \cite{BMR} for the algebraically closed case and \cite{GIT}, \cite{cochar} for arbitrary $k$.  Note that if ${\mathbf h}\in H^m$ is a generic tuple for $H$ then the centralizer $C_G(H)$ coincides with the stabilizer $G_{\mathbf h}$.

\begin{prop}
\label{prop:ratgeom}
 Let $H$ be a $k$-subgroup of $G$.  Suppose $k$ is perfect.  Then $H$ is $G$-completely reducible over $k$ if and only if $H$ is $G$-completely reducible.
\end{prop}

\begin{proof}
 If $k$ is perfect then $\ovl{k}/k$ is separable and $C_G(H)$ is $k$-defined.  The result now follows from
% Theorem~\ref{thm:GIT_crit} together with
 \cite[Cor.\ 9.7(i)]{cochar}.
\end{proof}

\begin{cor}
\label{cor:char0}
Suppose ${\rm char}(k)= 0$.   Let $H$ be a $k$-subgroup of $G$.  Then $H$ is $G$-completely reducible over $k$ if and only if $H$ is reductive.
\end{cor}

\begin{proof}
 If $k= \ovl{k}$ then this is well known (see \cite[Prop.\ 4.2]{serre2}, \cite[\S 2.2, \S 6.3]{BMR}, for example).  The result for arbitrary $k$ now follows from Proposition~\ref{prop:ratgeom}.
\end{proof}

Recall that if $S$ is a $k$-split torus of $G$, then $C_G(S)$ is an R-Levi $k$-subgroup of $G$ \cite[Lem. 2.5]{cochar}.
Part (i) of the next result gives the converse, and part (ii) strengthens \cite[Cor.~9.7(ii)]{cochar}: we do not need the hypotheses that $H$ and $C_G(H)$ are $k$-defined.
See also \cite[Prop.~3.2]{serre2}.

\begin{prop}
\label{prop:levi_ascent_descent}
Let $L$ be an R-Levi $k$-subgroup of $G$ and let $H$ be a subgroup of $L$.
\begin{itemize}
\item[(a)] There exists a $k$-split torus $S$ in $G$ such that $L = C_G(S)$.
\item[(b)] $H$ is $G$-completely reducible over $k$ if and only if $H$ is $L$-completely reducible over $k$.
\end{itemize}  
\end{prop}

\begin{proof}
(a).  We can choose $\lambda\in Y_{k_s}(G)$ such that $L= C_G({\rm Im}(\lambda))$.  
Let $\lambda= \lambda_1, \lambda_2,\ldots, \lambda_r\in Y_{k_s}(G)$ be the ${\rm Gal}(k_s/k)$-conjugates of $\lambda$ and let $S$ be the subtorus of $Z(L)^0$ generated by the subtori ${\rm Im}(\lambda_i)$.  Then $S$ is $k$-defined and $L= C_G(S)$.  The product map $\lambda_1\times\cdots \times \lambda_r$ gives an epimorphism from $\ovl{k}^*\times\cdots \times \ovl{k}^*$ onto $S$.  But a quotient of a split $k$-torus is $k$-split \cite[III.8.4 Cor.]{borel}, so $S$ is split.  

(b). Given (a), the result now follows from Theorem~\ref{thm:GIT_crit} together with \cite[Thm.\ 5.4(ii)]{cochar}.
\end{proof}

We finish the section with some results involving non-connected reductive groups which are needed in the sequel.  
Note that if $Q$ is an R-parabolic $k$-subgroup of $G$ and $M$ is an R-Levi $k$-subgroup of $Q$ then $Q^0$ is a parabolic $k$-subgroup of $G^0$ and $M^0$ is a Levi $k$-subgroup of $Q^0$; see \cite[Sec.\ 6]{BMR}.

\begin{lem}
\label{lem:Levi_tor}
 Let $P$ be an R-parabolic subgroup of $G$ and let $T$ be a maximal torus of $P$.  Then there is a unique R-Levi subgroup $L$ of $P$ such that $T\subseteq L$.  If $P$ and $T$ are $k$-defined then $L$ is $k$-defined.
\end{lem}

\begin{proof}
 The first assertion is \cite[Cor.\ 6.5]{BMR}.  For the second, suppose $P$ and $T$ are $k$-defined. 
 Then the unique R-Levi subgroup $L$ of $P$ containing $T$ must be Galois-stable and hence $k$-defined also. 
 \end{proof}

\begin{lem}
\label{lem:conn_to_nonconn}
 \begin{enumerate}
  \item[(a)] Let $Q$ be an R-parabolic $k$-subgroup of $G$ and set $P= Q^0$.  Then the R-Levi $k$-subgroups of $Q$ are precisely the subgroups of the form $N_Q(L)$ for $L$ a Levi $k$-subgroup of $P$. %\smallskip\\
  \item[(b)] Let $Q$, $P$ be as in (a) and let $H$ be a subgroup of $P$.  Then $H$ is contained in an R-Levi $k$-subgroup of $Q$ if and only if $H$ is contained in a Levi $k$-subgroup of $P$.  Moreover, if $L$ is a Levi $k$-subgroup of $P$ then $c_{N_Q(L)}(H)$ is $N_Q(L)$-completely reducible over $k$ if and only if $c_L(H)$ is $L$-completely reducible over $k$. %\smallskip\\
  \item[(c)] Let $H$ be a subgroup of $G^0$.  Then $H$ is $G$-completely reducible over $k$ if and only if $H$ is $G^0$-completely reducible over $k$.
 \end{enumerate}
\end{lem}

\begin{proof}
 (a) As observed above, if $M$ is an R-Levi subgroup of $Q$ then $M^0$ is a Levi subgroup of $P$, and $N_Q(M^0)^0= N_P(M^0)^0= M^0$.  Let $L$ be a Levi subgroup of $P$ and let $T$ be a maximal torus of $L$.  By Lemma~\ref{lem:Levi_tor} there is a unique R-Levi subgroup $M$ of $Q$ such that $T\subseteq M$.  The Levi subgroups $M^0$ and $L$ of $P$ both contain $T$, so by Lemma~\ref{lem:Levi_tor} they are equal; in particular, $M$ normalizes $L$.  Now $N_Q(T)$ normalizes $L$ by Lemma~\ref{lem:Levi_tor}, so $N_Q(L)$ meets every component of $Q$.  Since $Q= M\ltimes R_u(Q)$, $M$ also meets every component of $Q$.  It follows that $M= N_Q(L)$.  Finally, $L$ contains a maximal $k$-torus of $P$ if and only if $N_Q(L)$ does, so $L$ is $k$-defined if and only if $N_Q(L)$ is, by Lemma~\ref{lem:Levi_tor}.\smallskip\\
 (b) The first assertion follows immediately from (a), and part (c) now follows.  For the second assertion of (b), note that the restriction of $c_{N_Q(L)}(H)$ to $P$ is $c_L$; the desired result now follows from part (c) applied to the reductive $k$-group $N_Q(L)$.
\end{proof}

\section{$k$-semisimplification}
\label{sec:ss}

Now we come to our main definition.

\begin{defn}
\label{defn:ss}
 Let $H$ be a subgroup of $G$.  We say that a subgroup $H'$ of $G$ is a \emph{$k$-semisimplification of $H$ (for $G$)} if there exist an R-parabolic $k$-subgroup $P$ of $G$ and an R-Levi $k$-subgroup $L$ of $P$ such that $H\subseteq P$, $H'= c_L(H)$ and $H'$ is $G$-completely reducible (or equivalently by Proposition~\ref{prop:levi_ascent_descent}(ii), $L$-completely reducible) over $k$.  We say \emph{the pair $(P,L)$ yields $H'$}.
\end{defn}

\begin{rems}
\label{rem:ss}
 \begin{enumerate}
  \item[(a)] Let $H$ be a subgroup of $G$.  If $H$ is $G$-cr over $k$ then clearly $H$ is a $k$-semisimplification of itself, yielded by the pair $(G,G)$.  
  \item[(b)] Suppose $(P,L)$ yields a $k$-semisimplification $H'$ of $H$.  Let $L_1$ be another R-Levi $k$-subgroup of $P$.  Then $L_1= uLu^{-1}$ for some $u\in R_u(P)(k)$, so $c_{L_1}(H)= uc_L(H)u^{-1}$.  Hence $(P,L_1)$ also yields a $k$-semisimplification of $H$.  We say that \emph{$P$ yields a $k$-semisimplification of $H$}.
  \item[(c)] It is straightforward to check that if $\phi$ is an automorphism of $G$ (as a $k$-group), $H$ is a subgroup of $G$ and $(P,L)$ yields a $k$-semisimplification $H'$ of $H$ then $\phi(H')$ is a $k$-semisimplification of $\phi(H)$, yielded by $(\phi(P), \phi(L))$.
  \item[(d)] For $G$ connected and $H$ a subgroup of $G(k)$, 
 Definition~\ref{defn:ss} recovers Serre's ``$G$-analogue" of a semisimplification from \cite[\S 3.2.4]{serre2}.  For $k= \ovl{k}$, Definition~\ref{defn:ss} generalizes the definition of ${\mathcal D}(H)$ following \cite[Lem.\ 4.1]{genred}.
 \end{enumerate}
\end{rems}

\begin{rem}
 Let ${\mathbf h}= (h_1,\ldots, h_m)\in H^m$ be a generic tuple for $H$.  Note that $c_\lambda$ extends in the obvious way to a homomorphism from a parabolic subalgebra ${\mathcal P}_\lambda$ of $M_n$ onto a Levi subalgebra ${\mathcal L}_\lambda$ of ${\mathcal P}_\lambda$, and ${\mathcal P}_\lambda$ contains the subalgebra ${\mathcal A}$ generated by $H$.  Since the elements $h_i$ generate ${\mathcal A}$, the elements $c_\lambda(h_i)$ generate $c_\lambda({\mathcal A})$.  But $c_\lambda({\mathcal A})$ is the subalgebra of ${\mathcal L}_\lambda$ generated by $c_\lambda(H)$, so we deduce that $c_\lambda({\mathbf h})= (c_\lambda(h_1),\ldots, c_\lambda(h_m))$ is a generic tuple for $c_\lambda(H)$.  Hence by Theorem~\ref{thm:GIT_crit}, $c_\lambda(H)$ is a $k$-semisimplification of $H$ if and only if $G(k)\cdot c_\lambda({\mathbf h})$ is cocharacter-closed over $k$.  It follows from Theorem~\ref{thm:RHMT} that $H$ admits at least one $k$-semisimplification: for we can choose $\lambda\in Y_k(G)$ such that $G(k)\cdot c_\lambda({\mathbf h})$ is cocharacter-closed over $k$, so $c_\lambda(H)$ is a $k$-semisimplification of $H$, yielded by $(P_\lambda, L_\lambda)$.
\end{rem}

\begin{lem}
\label{lem:k-cochar}
Suppose that $H'$ is a $k$-semisimplification of $H$.
Then there is $\lambda\in Y_k(G)$ such that $H'$ is yielded by the pair $(P_\lambda,L_\lambda)$.
\end{lem}

\begin{proof}
Suppose $H'$ is yielded by the pair $(P,L)$.  
By the discussion in Section~\ref{sec:cochar}, there exist a maximal $k$-torus $T$ of $L$ and $\mu\in Y_{k_s}(T)$ such that $P= P_\mu$ and $L= L_\mu$.  
Choose a finite Galois extension $k'/k$ such that $T$ splits over $k'$, and let $\lambda= \sum_{\gamma\in {\rm Gal}(k'/k)} \gamma\cdot \mu\in Y_k(T)$.  
One checks easily that $H\subseteq P_\lambda$ and $c_\lambda|_H= c_\mu|_H$ (cf.\ the proof of \cite[Lem.~2.5(ii)]{GIT}).  Hence $(P_\lambda, L_\lambda)$ also yields $H'$.
\end{proof}

Here is our main result, which was proved in the special case $k= \ovl{k}$ in \cite[Prop.~5.14(i)]{GIT}, cf.\ \cite[Prop.~3.3(b)]{serre2}. 
The uniqueness asserted in Theorem~\ref{thm:main} is akin to the theorem of Jordan--H\"older.

\begin{thm}
\label{thm:main}
 Let $H$ be a subgroup of $G$.  Then any two $k$-semisimplifications of $H$ are $G(k)$-conjugate.
\end{thm}

\begin{proof}
 Let $H_1, H_2$ be $k$-semisimplifications of $H$.  By Lemma~\ref{lem:k-cochar}, there exist $\lambda_1, \lambda_2\in Y_k(G)$ such that $(P_{\lambda_1}, L_{\lambda_1})$ realizes $H_1$ and $(P_{\lambda_2}, L_{\lambda_2})$ realizes $H_2$.  Let ${\mathbf h}\in H^m$ be a generic tuple for $H$.  Then $c_{\lambda_i}({\mathbf h})$ is a generic tuple for $H_i$ for $i= 1,2$, and each orbit $G(k)\cdot c_{\lambda_i}({\mathbf h})$ is cocharacter-closed over $k$ and accessible from ${\mathbf h}$ over $k$ (Example~\ref{ex:accessible_tuple}).  It follows from the uniqueness result in Theorem~\ref{thm:RHMT} that the closed subset $C_{{\mathbf h}}:= \{g\in G\,|\,g\cdot c_{\lambda_1}({\mathbf h})= c_{\lambda_2}({\mathbf h})\}$ contains a $k$-point.
 
 Pick $g\in C_{{\mathbf h}}$.  If $H_2= gH_1g^{-1}$ then we are done.  Otherwise there exists $h\in H$ such that $gc_{\lambda_1}(h)g^{-1}\not\in H_2$ or $g^{-1}c_{\lambda_2}(h)g\not\in H_1$.  Without loss assume the former.  We can repeat the above argument, replacing ${\mathbf h}$ with the generic tuple ${\mathbf h}':= ({\mathbf h}, h)\in H^{m+1}$; note that $C_{{\mathbf h}'}$ is properly contained in $C_{{\mathbf h}}$.  The result now follows by a descending chain condition argument.
 %by Remark \ref{rem:generic-tuple}.
\end{proof}

\begin{defn}
 We define ${\mathcal D}_k(H)$ to be the set of $G(k)$-conjugates of any $k$-semisimplification of $H$ (cf.\ the discussion preceding \cite[Thm.\ 1.4]{genred}).  This is well-defined by Theorem~\ref{thm:main}.
\end{defn}

\begin{ex}\label{ex:comments}
 Let $H$ be a subgroup of $G$.  As noted in Remark \ref{rem:ss}(a), if $H$ is $G$-cr over $k$ then $H$ is a $k$-semisimplification of itself, yielded by the pair $(G,G)$.  
 If $H$ is a $G$-ir subgroup of $G$, then $H$ is the only $k$-semisimplification of $H$: this shows that not every element of ${\mathcal D}_k(H)$ need be a $k$-semisimplification of $H$.  
 In a similar vein, if $P$ and $Q$ are arbitrary R-parabolic $k$-subgroups of $G$ and $Q\supseteq P$ then it is easily seen that $Q$ yields a $k$-semisimplification of $P$ if and only if $P^0= Q^0$.
\end{ex}

\begin{ex} 
 Let $H$ be a subgroup of $G$ and let $P$ be minimal among the R-parabolic $k$-subgroups that contain $H$.  Let $L$ be an R-Levi $k$-subgroup of $P$.  We claim that $c_L(H)$ is $L$-ir over $k$ (cf.\ \cite[Prop.~3.3(a)]{serre2} and \cite[Sec.~3]{BMR}); it then follows from Proposition~\ref{prop:levi_ascent_descent}(ii) that $c_L(H)$ is a $k$-semisimplification of $H$.  Suppose $c_L(H)$ is not $L$-ir: say, $c_L(H)\subseteq Q$, where $Q$ is a proper R-parabolic $k$-subgroup of $L$.  There exist a maximal $k$-torus $T$ of $Q$ and cocharacters $\lambda, \mu\in Y_{k_s}(T)$ such that $P= P_\lambda$, $L= L_\lambda$ and $Q= P_\mu$.  Now $H\subseteq QR_u(P)\subsetneq P$, and clearly $QR_u(P)$ is $k$-defined.  But it is easily checked that $QR_u(P)= P_{m\lambda+\mu}$ for suitably large $m\in {\mathbb N}$ (cf.\ \cite[Lem.~6.2(i)]{BMR}), so $QR_u(P)$ is an R-parabolic $k$-subgroup of $G$, contradicting the minimality of $P$.  Conversely, if $P$ is an R-parabolic $k$-subgroup with R-Levi $k$-subgroup $L$ such that $P\supseteq H$ and $c_L(H)$ is $L$-ir over $k$ then a similar argument shows that $P$ is minimal among the R-parabolic $k$-subgroups containing $H$.  This proves the claim.
 
 In particular, let $G$, $H$, $\lambda$ and $H'$ be as in the $\GL_n$ example in Section~\ref{sec:intro}.  Let $P= P_\lambda$ be the parabolic subgroup of block upper triangular matrices with blocks of size $n_1,\ldots, n_r$ down the leading diagonal.  Let $L= L_\lambda$ be the subgroup of block diagonal matrices with blocks of size $n_1,\ldots, n_r$ down the leading diagonal.  Since each $n_i\times n_i$ block yields an irreducible representation of $H':= c_\lambda(H)$, $H'$ is $L$-ir over $k$, so $P$ is minimal among the R-parabolic $k$-subgroups of $G$ containing $H$; hence $H'$ is the $k$-semisimplification of $H$ yielded by $(P,L)$.
\end{ex}

\begin{ex}
 Suppose ${\rm char}(k)= 0$.  Let $H$ be a $k$-subgroup of $G$ and let $P$ be an R-parabolic subgroup of $G$ with R-Levi subgroup $L$ such that $P\supseteq H$.  Then Corollary~\ref{cor:char0} implies that $c_L(H)$ is a $k$-semisimplification of $H$ if and only if $R_u(H)\subseteq R_u(P)$.
\end{ex}

\begin{rem}
 Given a reductive $k$-group $G$ and a subgroup $H$ of $G$, we may (as in Remark~\ref{rem:absolute}) regard $G$ as a $\ovl{k}$-group by forgetting the $k$-structure, so it makes sense to consider the semisimplification (i.e., the $\ovl{k}$-semisimplification) of $H$.  The reader is warned that it can happen that $H$ is $G$-cr over $k$ but not $G$-cr, or vice versa (see \cite[Ex.~5.11]{BMR} and \cite[Ex.~7.22]{BMRT}), so there is no direct relation between the notions of $k$-semisimplification and semisimplification.
\end{rem}

\section{Optimality and normal subgroups}
\label{sec:opt}

In Example \ref{ex:comments} we observed that not every element of ${\mathcal D}_k(H)$ need be a $k$-semisimplification of $H$.  On the other hand, it can happen that $H$ is contained in many different R-parabolic subgroups of $G$, and there may exist many conjugate, but different, $k$-semisimplifications.  We now recall two constructions that give under some extra hypotheses a more canonical choice of R-parabolic subgroup yielding a $k$-semisimplification.  They apply in particular when $G= \GL_n$ (see Example \ref{ex:GL_n_geom}); this does not seem to be well known even when $k= \ovl{k}$.

\smallskip
\noindent {\bf First construction:} Suppose $G$ is connected, $H$ is a subgroup of $G$ and $H$ is not $G$-cr over $k$.  We use the theory of spherical buildings (see \cite{serre1.5}, \cite{serre2}) and the argument of \cite[Proof of Thm.\ 1.1]{BMR:sepext}.  Recall that the spherical building $\Delta_k(G)$ of $G$ is a simplicial complex whose simplices are the  parabolic $k$-subgroups of $G$, ordered by reverse inclusion (the proper $k$-parabolic subgroups correspond to the non-empty simplices).  The apartments of $\Delta_k(G)$ are the sets of all $k$-parabolic subgroups of $G$ that contain a fixed maximal split $k$-torus $S$ of $G$.  The set $\Sigma$ of parabolic $k$-subgroups $P$ of $G$ such that $P\supseteq H$ is a convex subcomplex of $\Delta_k(G)$, and $\Sigma$ is not completely reducible in the sense of \cite[\S 2.2]{serre2} because $H$ is not $G$-cr over $k$ (see \cite[\S 3.2.1]{serre2}).  By the Tits Centre Conjecture---see, e.g., \cite[\S 2.6]{stcc}, and \cite[\S 2.4]{serre2} and the references therein---$\Sigma$ has a so-called ``centre'': a proper parabolic $k$-subgroup $P_c\in \Sigma$ such that $P_c$ is fixed by any building automorphism of $\Delta_k(G)$ that stabilizes $\Sigma$.  In particular, $P_c$ is stabilized by any $k$-automorphism of $G$ that stabilizes $H$.

\begin{lem}
\label{lem:minlopt}
 Let $G$, $H$ and $\Sigma$ be as above.  Let $P_c$ be a centre for $\Sigma$ such that $P_c$ is not properly contained in any other centre for $\Sigma$.  Then $P_c$ yields a $k$-semisimplification of $H$.
\end{lem}

\begin{proof}
 Let $\Lambda$ be the set of $k$-parabolic subgroups $Q$ of $G$ such that $Q\subseteq P_c$.  Fix a Levi $k$-subgroup $L$ of $P_c$.  We have an inclusion-preserving bijection $\psi$ from $\Lambda$ to $\Delta_k(L)$ given by $Q\mapsto Q\cap L$, with inverse given by $R\mapsto R R_u(P_c)$.  Let $\Sigma_L$ be the subset of $\Delta_k(L)$ consisting of all the $k$-parabolic subgroups of $L$ that contain $c_L(H)$.  It is clear that $\psi(\Sigma\cap \Lambda)= \Sigma_L$.  If $\phi$ is a building automorphism of $\Delta_k(G)$ that fixes $P_c$ then $\phi$ stabilizes $\Lambda$, and we get an automorphism $\phi_L$ of $\Delta_k(L)$ (as a simplicial complex) given by $\phi_L(Q\cap L)= \phi(Q)\cap L$; moreover, if $\phi$ stabilizes $\Sigma$ then $\phi_L$ stabilizes $\Sigma_L$.
 
 We claim that $\phi_L$ is a building automorphism of $\Delta_k(L)$.  It is enough to show that $\phi_L$ maps apartments to apartments.  Let $S$ be a maximal split $k$-torus of $L$ (and hence of $G$).  Since $\phi$ is a building automorphism, there is a maximal split $k$-torus $S'$ of $G$ such that for every $k$-parabolic subgroup $Q$ of $G$ that contains $S$, $\phi(Q)$ contains $S'$.  In particular, $S'\subseteq P_c$ since $\phi(P_c)= P_c$.  By Lemma~\ref{lem:Levi_tor} there is a $k$-Levi subgroup $L'$ of $P_c$ such that $S'\subseteq L'$.  By Lemma~\ref{lem:parprops}(ii) there exists $u\in R_u(P_c)(k)$ such that $uS'u^{-1}\subseteq L$.  Let $R\in \Delta_k(L)$ such that $S\subseteq R$: say, $R= Q\cap L$ for $Q\in \Lambda$.  Then $S'\subseteq \phi(Q)$.  Since $\phi(Q)\subseteq P_c$, $R_u(\phi(Q))$ contains $R_u(P_c)$, so $uS'u^{-1}\subseteq \phi(Q)$.  Hence $uS'u^{-1}\subseteq \phi(Q)\cap L= \phi_L(R)$.  This proves the claim.
 
 Now suppose $P_c$ does not yield a $k$-semisimplification of $H$.  Then $c_L(H)$ is not $L$-cr over $k$.  By the discussion before the lemma, $\Sigma_L$ has a centre $R\subsetneq L$.  We have $R= Q\cap L$ for some $Q\in \Lambda$ with $Q\subsetneq P_c$.  But the results in the previous paragraph imply that $Q$ is a centre for $\Sigma$, contradicting the minimality of $P_c$.
\end{proof}

\smallskip
\noindent {\bf Second construction:} We allow $G$ to be non-connected again.  Suppose the following property holds for a subgroup $H$ of $G$:\\

\noindent ($*$) there exists an R-parabolic $k$-subgroup $P$ of $G$ such that $H\subseteq P$ but $H$ is not contained in any R-Levi subgroup---that is, any R-Levi $\ovl{k}$-subgroup---of $P$.\\

This hypothesis implies in particular that $H$ is not $G$-cr over $k$.  The construction in \cite[Sec.~5.2]{GIT} then
yields a canonical so-called `optimal destabilising' R-parabolic $k$-subgroup $P_{\rm opt}$ of $G$ such that $H\subseteq P_{\rm opt}$ but $H$ is not contained in any R-Levi subgroup of $P_{\rm opt}$.
If $k$ is perfect then $P_{\rm opt}$ yields both a $\ovl{k}$-semisimplification of $H$ and a $k$-semisimplification of $H$ by \cite[Thm.~4.2]{kempf}, but both can fail for general $k$.
Moreover, $P_{\rm opt}$ is stabilized by any $k$-automorphism of $G$ that stabilizes $H$; in particular, if $M$ is a $k$-subgroup of $G$ that normalizes $H$ then $M(k)$ normalizes $P_{\rm opt}$. See \cite[Thm.~5.16]{GIT} for details.

This construction rests on the notion of an ``optimal destabilising cocharacter'' due to work of Hesselink \cite{He}, Kempf \cite{kempf} and Rousseau \cite{rousseau}.  Roughly speaking, the idea is as follows.  Take a generic tuple ${\mathbf h}\in H^m$ for $H$.  Choose ${\mathbf g}\in G^m$ such that $G(k)\cdot {\mathbf g}$ is accessible from ${\mathbf h}$ over $k$ and $G(k)\cdot {\mathbf g}$ is cocharacter-closed over $k$.  Set ${\mathcal O}({\mathbf h})= G(\ovl{k})\cdot {\mathbf g}$; note that ${\mathcal O}({\mathbf h})$ is uniquely defined by Theorem~\ref{thm:RHMT}.  Roughly speaking, we define $\lambda_{\rm opt}\in Y_k(G)$ to be the cocharacter that takes ${\mathbf h}$ into ${\mathcal O}({\mathbf h})$ as quickly as possible (in an appropriate sense), and we define $P_{\rm opt}$ to be $P_{\lambda_{\rm opt}}$.  (In fact, we need a slight variation---due to Hesselink---on this construction: rather than taking a single generic tuple ${\mathbf h}$, one considers the action of a cocharacter $\lambda$ on all elements of $H$ at once.)  Note that $P_{\rm opt}$ is not uniquely determined (see \cite[Rem.~5.22]{GIT}).

Now suppose that $H$ is a subgroup of $G$ such that $C_G(H)$ is $k$-defined.
One can show that if $H$ is $G$-cr then $H$ is $G$-cr over $k$ (as previously noted, the converse is false).
In fact, we prove a slightly stronger result: if $H$ is not $G$-cr over $k$ then hypothesis ($*$) holds.  To see this, choose a generic tuple ${\mathbf h}\in H^m$.  We can find $\lambda\in Y_k(G)$ such that $(P_\lambda, L_\lambda)$ yields a $k$-semisimplification $H'$ of $H$; so $G(k)\cdot c_\lambda({\mathbf h})$ is cocharacter-closed over $k$ but $G(k)\cdot {\mathbf h}$ is not.  If $H$ is contained in an R-Levi $\ovl{k}$-subgroup $L$ of $P_\lambda$ then $c_\lambda({\mathbf h})= u\cdot {\mathbf h}$ for some $u\in R_u(P_\lambda)$.  But then \cite[Thm.~7.1]{cochar} implies that $c_\lambda({\mathbf h})= u_1\cdot {\mathbf h}$ for some $u_1\in R_u(P_\lambda)(k)$, so $G(k)\cdot c _\lambda({\mathbf h})= G(k)\cdot {\mathbf h}$, a contradiction.

\begin{rem}
	\label{rem:norm}
	Let $M$ be a $k$-subgroup of $G$ such that $M$ normalizes $H$, and let $P$ be the R-parabolic subgroup of $G$ obtained from one of the constructions above.
	Then it is automatic that $M(k)$ normalizes $P$.
	However, under the extra hypothesis that $H$ is $k$-defined, we can in fact show that $M\subseteq N_G(P)$.
	To see this, one can first extend the field from $k$ to $k_s$ and then show that the R-parabolic subgroup obtained from either of the constructions is $k$-defined (cf.\ \cite[Proof of Thm.~1.1]{BMR:sepext} and \cite[Sec.~4]{kempf}), and hence coincides with $P$---this implies that $M(k_s)$, and hence $M$, normalizes $P$.
\end{rem}

\begin{rem}
	\label{rem:smooth_crit}
	There are some limitations on the constructions given above.
	First, without the hypothesis that $k$ is perfect, it can happen that the subgroup obtained from $P_{\rm opt}$ is not $G$-cr over $k$, and is therefore not a $k$-semisimplification of $H$.  (It is, however, $G(\ovl{k})$-conjugate to a $k$-semisimplification of $H$.)
	Second, as yet there is no theory of optimal destabilising subgroups that holds for arbitrary fields---this means that we do not know how to define a version of $P_{\rm opt}$ for a subgroup $H$ that is not $G$-cr over $k$ if ($*$) does not hold.
	See \cite[Sec.~1 and Ex.~5.21]{GIT} for further discussion of this latter point.
\end{rem}

By combining the two constructions above we obtain the following ``Clifford theory'' result, exploring the link between the semisimplification of a group and a normal subgroup.
In the case $k$ is algebraically closed, part (a) is \cite[Thm.~3.10]{BMR}.

\begin{thm}
	\label{thm:normal}
	Let $M$ be a $k$-subgroup of $G$ and let $H$ be a normal $k$-subgroup of $M$.  Suppose at least one of the following holds:
	\begin{enumerate}
		\item[(i)]
		$k$ is perfect.
		\item[(ii)] $G$ is connected.
	\end{enumerate}
	Then:
	\begin{enumerate}
		\item[(a)] If $M$ is $G$-completely reducible over $k$ then $H$ is $G$-completely reducible  over $k$.
		\item[(b)] There is an R-parabolic subgroup $P$ of $G$ such that $M\subseteq P$ and $P$ yields both a $k$-semisimplification of $M$ and a $k$-semisimplification of $H$.  In particular, there exist $k$-semisimplifications $M'$ (resp., $H'$) of $M$ (resp., of $H$) such that $H'$ is normal in~$M'$.
	\end{enumerate}
\end{thm}

\begin{proof}
	Suppose $H$ is not $G$-cr over $k$.  Choose $P= P_{\rm opt}$ in case (i) and $P= P_c$ in case (ii).  Then $M\subseteq N_G(P)$ by Remark~\ref{rem:norm}.  Since $H$ is not contained in any R-Levi $k$-subgroup of $P$, $H$ is not contained in any R-Levi $k$-subgroup of $N_G(P)$ (Lemma~\ref{lem:conn_to_nonconn}).  Hence $M$ is not contained in any R-Levi $k$-subgroup of $N_G(P)$.  It follows that $M$ is not $G$-cr over $k$.  This proves part (a).
	
	For (b), pick $\lambda\in Y_k(G)$ such that $(P_\lambda, L_\lambda)$ yields a semisimplification $M':= c_\lambda(M)$ of $M$.  Then $c_\lambda(M)$ is $G$-cr over $k$ and $c_\lambda(H)$ is normal in $ c_\lambda(M)$.  Now $c_\lambda(M)$ and $c_\lambda(H)$ satisfy the hypotheses of the theorem,
	so $c_\lambda(H)$ is $G$-cr over $k$ by (a).  Hence $(P_\lambda, L_\lambda)$ yields a semisimplification $H':= c_\lambda(H)$ of $H$ as well, and $H'$ is normal in $ M'$.
\end{proof}

\begin{rem}
	The hypothesis in part (ii) can be weakened: one only needs to assume that $H\subseteq G^0$.
	In order to make the proof go through, one needs to verify that the first construction above extends to this situation.
\end{rem}

\begin{ex}
	\label{ex:GL_n_geom}
	Let $H$ be a $k$-subgroup of $G = \GL_n$ such that $H$ is not completely reducible over $k$.  Since $H$ is separable, $C_G(H)$ is $k$-defined, so $H$ is not $G$-completely reducible; we obtain a parabolic $k$-subgroup $P_{\rm opt}$ as above which yields a subgroup $H'$.  We claim that $H'$ is a $k$-semisimplification of $H$.  For suppose $H'$ is not $G$-cr over $k$.  Choose ${\mathbf h}$, ${\mathbf g}$ as above, and let ${\mathbf h}'= c_{\lambda_{\rm opt}}({\mathbf h})$ (so that ${\mathbf h}'$ is a generic tuple for $H'$).  Since $C_G(H')$ is $k$-defined, hypothesis $(*)$ holds, so we obtain an optimal cocharacter which takes ${\mathbf h}'$ out of $G\cdot {\mathbf h}'= {\mathcal O}({\mathbf h})$ and into ${\mathcal O}({\mathbf h}')$.  But ${\mathbf g}$ is accessible from ${\mathbf h}'$ over $k$ by \cite[Thm.~4.3(ii)]{cochar}, so ${\mathcal O}({\mathbf h}')= {\mathcal O}({\mathbf h})$, a contradiction.
	
	The parabolic subgroup $P_{\rm opt}$ is the stabilizer of some flag ${\mathcal F}$ of subspaces of $k^n$, and ${\mathcal F}$ does not admit a complementary $H$-stable flag of subspaces of $k^n$. By Remark~\ref{rem:norm}, $C_G(H)$ is a subgroup of $P_{\rm opt}$---that is, $C_G(H)$ stabilizes ${\mathcal F}$---and likewise the normalizer $N_G(H)$ stabilizes ${\mathcal F}$ if $N_G(H)$ is $k$-defined.  If $k$ is perfect then $N_G(H)$ is automatically $k$-defined but it need not be $k$-defined in general; see \cite{HeSt} for further discussion.
\end{ex}

\begin{rem}
 Hesselink gives an example \cite[(8.5)~Ex.]{He} of a subgroup $H$ of an almost simple group $G$ of type $C_2$ such that $P_{\rm opt}$ is not a minimal centre for $\Sigma$, the subcomplex of the building $\Delta_k(G)$ of $G$ consisting of all parabolic subgroups of $G$ that contain $H$.  This shows that the two constructions above can yield different R-parabolic subgroups.  Nevertheless, the corresponding $k$-semisimplifications of $H$ are 
 $G(k)$-conjugate, thanks to Theorem \ref{thm:main}. 
 \end{rem}

%%%%%%%%%%%%%%%%%%%%%%%%%%%%%%%%%%%%%%%%%%%%%%%%%%%%%%%%%%%%%%%%%%%%%%
%%%%%%%%%%%%% Acknowledgements
%%%%%%%%%%%%%%%%%%%%%%%%%%%%%%%%%%%%%%%%%%%%%%%%%%%%%%%%%%%%%%%%%%%%%%

\medskip
\noindent {\bf Acknowledgements}:

\noindent We are grateful to Brian Lawrence for his questions, which motivated this paper, and for his comments on an earlier draft.
We thank the referee for their comments.

%\bigskip
%%%%%%%%%%%%%%%%%%%%%%%%%%%%%%%%%%%%%%%%%%%%%%%%%%%%%%%%%%%%%%%%%%%%%%
%%%%%%%%%%%%% bibliography
%%%%%%%%%%%%%%%%%%%%%%%%%%%%%%%%%%%%%%%%%%%%%%%%%%%%%%%%%%%%%%%%%%%%%%

\end{document}